\documentclass[11pt]{article}
\usepackage[T1]{fontenc}
\usepackage{lmodern,amsmath,amssymb,amsthm,mathtools}
\usepackage[margin=1in]{geometry}
\usepackage{microtype}
\usepackage[hidelinks]{hyperref}
\hypersetup{pdftitle={Dihedral fields with Bloch lattice index five and quadratic subfield Q(i)},pdfauthor={Zheng Huimin},pdfcreator={},pdfproducer={}}
\ifdefined\pdfinfoomitdate\pdfinfoomitdate=1\fi
\ifdefined\pdftrailerid\pdftrailerid{}\fi
\ifdefined\pdfsuppressptexinfo\pdfsuppressptexinfo=15\fi
\newtheorem{theorem}{Theorem}
\newtheorem{lemma}[theorem]{Lemma}
\newtheorem{corollary}[theorem]{Corollary}
\newtheorem{proposition}[theorem]{Proposition}
\theoremstyle{remark}
\newcommand{\Q}{\mathbb Q}\newcommand{\Z}{\mathbb Z}
\newcommand{\F}{\mathbb F}\newcommand{\Gal}{\operatorname{Gal}}
\newcommand{\res}{\operatorname{res}}\newcommand{\cor}{\operatorname{cor}}

\newcommand{\coker}{\operatorname{coker}}
\newcommand{\N}{\operatorname{N}}

\title{Dihedral fields with Bloch lattice index five\protect\\and quadratic subfield \(\Q(i)\)}
\author{Zheng Huimin\\[0.4em]
  \small Anhui Science and Technology University\\
  \small \texttt{zhenghm@ahstu.edu.cn}}
\date{October 16, 2026}
\begin{document}
\maketitle
\begin{abstract}
We construct infinitely many Galois extensions \(M/\Q\) with dihedral
group of order ten, quadratic subfield \(\Q(i)\), and Bloch lattice index
\(Q_2(M)=5\). More precisely, we specify a Chebotarev set of rational primes
of density \(1/50\). For every prime \(\ell\) in this set, at least two of
four mixed cyclic quintic extensions of \(\Q(i)\) of conductor \((25\ell)\)
have this index and satisfy \(K_2(O_M)(5)\simeq(\Z/5)^2\). A
two-dimensional cup-product pencil selects the extensions. Generalized
Bockstein maps, integral co-descent, and a motivic genus formula then
determine the tame kernel and the subfield Bloch-lattice index. The result
is unconditional.
\end{abstract}

\section{Introduction and main results}
Throughout, \(D_{10}\) denotes the dihedral group of order ten, and \(O_J\)
is the ring of integers of a number field \(J\). The group \(G_{J,S}\) is
the Galois group of the maximal extension of \(J\) unramified outside
\(S\); all cohomology is continuous. The symbols \(\F_5(m)\) and
\(\Z_5(m)\) denote Tate twists. A character of a ray class group of an
imaginary quadratic field is called anti-invariant if complex conjugation
acts on it by \(-1\).

Brauer--Kuroda relations compare the zeta functions of a dihedral number
field and its proper subfields. Their regulator analogues introduce an
arithmetic index: the subfield Bloch lattices need not generate the full
Bloch lattice. Browkin--Gangl \cite{BG} studied this index and supplied
numerical evidence for a nontrivial value in degree ten. Guo--Qin
\cite{GQ} proved a divisibility theorem that restricts the index, in
degree ten, to \(1\) or \(5\). The arithmetic
task addressed here is to realize the nontrivial value in an infinite
family while keeping the quadratic subfield fixed.

We do this for \(p=5\) and quadratic subfield \(\Q(i)\). The construction
allows a choice among four mixed ray class fields for every prime in one
fixed Chebotarev set. A quadratic determinant ensures that at least two
choices work. An integral cohomology argument then proves the actual Bloch
lattice index, rather than inferring it from a numerical regulator quotient.

For a number field \(J\), let \(\mathcal B(J)\) denote the free quotient of its
Bloch group, where the free quotient means the quotient by all torsion;
equivalently, it is the free quotient of \(K_3^{\mathrm{ind}}(J)\).
The natural Suslin sequence, in the Bloch-group conventions of
\cite[Section 2]{GQ}, identifies these free quotients compatibly with field
embeddings. The regulator identifies
the free quotient with its Bloch lattice. Let \(M/\Q\) be a non-real
\(D_{10}\)-extension, let \(k\) be its quadratic subfield, and let
\(F_1,F_2\) be two different quintic reflection subfields. The index considered
by Browkin--Gangl and Guo--Qin is
\[
 Q_2(M)=[\mathcal B(M):\mathcal B(k)+\mathcal B(F_1)+\mathcal B(F_2)],
\]
where subgroups are understood via their natural images. Guo--Qin
\cite[Theorem 4.3]{GQ} prove \(Q_2(M)\mid5\).

Put \(k=\Q(i)\), \(E=k(\mu_5)=\Q(\zeta_{20})\), and let \(S_0\) consist of
the primes above five and the infinite places. All sets of primes are pulled
back to every extension in use. Write \(K_2(O_J)(5)\) for the five-primary
subgroup of \(K_2(O_J)\). The base vanishing used below is proved in Lemma~\ref{lem:base}.

For \(\ell\equiv1\pmod{20}\), choose nonzero anti-invariant
\(\F_5\)-valued ray class characters \(\eta_0\) modulo \((25)\) and
\(\eta_\ell\) modulo \((\ell)\). Each of these anti-invariant character
spaces is one-dimensional, as verified below. Pulling both characters
back to the common Galois group, define the four mixed extensions by
\[
 M_\lambda=\overline{k}^{\ker(\eta_0+\lambda\eta_\ell)},
 \qquad \lambda\in\F_5^\times.
\]
Changing either chosen generator only relabels these four extensions.

\begin{theorem}\label{thm:main}
There is a set \(\mathcal P\) of rational primes of natural and Dirichlet density
\(1/50\), all congruent to one modulo twenty, with the following property.
For each \(\ell\in\mathcal P\), at least two of the four anti-invariant mixed
cyclic quintic ray class extensions \(M/k\) of conductor \((25\ell)\) satisfy
\[
 \Gal(M/\Q)\simeq D_{10},\qquad Q_2(M)=5,
 \qquad K_2(O_M)(5)\simeq(\Z/5)^2.
\]
Their discriminants are
\(\operatorname{disc}(M)=-2^{10}5^{16}\ell^8\).
If \(a=v_5(\ell-1)\) and \(S=S_0\cup\{v:v\mid\ell\}\), then
\[
 H^2(G_{M,S},\Z_5(2))\simeq(\Z/5^{a+1})^2.
\]
In particular, there are infinitely many pairwise non-isomorphic such fields
with the same quadratic subfield and the same degree over \(\Q\).
\end{theorem}

The theorem permits a choice among four actual ray class fields for each
parameter. It does not assert that an arbitrary preselected ray has index five.
The density is a density of parameter primes, not a density ordered by field
discriminant.

There is nevertheless a direct lower bound when the fields are ordered by
discriminant. Let \(N(X)\) count the \(\Q\)-isomorphism classes of
\(D_{10}\)-fields \(M\) with quadratic subfield \(\Q(i)\),
\(Q_2(M)=5\), \(K_2(O_M)(5)\simeq(\Z/5)^2\), and
\(|\operatorname{disc}(M)|\leq X\).
\begin{corollary}\label{cor:count}
With \(C=2^{10}5^{16}\), one has
\[
 \liminf_{X\longrightarrow\infty}
 \frac{N(X)\log X}{X^{1/8}}\ \geq\ \frac{8}{25C^{1/8}}.
\]
\end{corollary}
\begin{proof}
Every \(\ell\in\mathcal P\) with \(\ell\leq(X/C)^{1/8}\)
contributes at least two fields. Distinct mixed characters give distinct
cyclic extensions over \(\Q(i)\); these are normal over \(\Q\), so
distinct such fields in a fixed algebraic closure are not
\(\Q\)-isomorphic. Distinct parameter primes give different
discriminants. Chebotarev and the prime number theorem give
\(\#\{\ell\in\mathcal P:\ell\leq Y\}\sim Y/(50\log Y)\).
Substituting \(Y=(X/C)^{1/8}\) proves the inequality.
\end{proof}

\paragraph{Proof strategy.}
The fixed Tate-line extension and a fixed conductor-25 dihedral seed define
the parameter set; their distinct Galois characters give density \(1/50\).
The seed cup product is invertible at the two new tame primes. Mixing the
seed with a tame character produces a two-dimensional matrix pencil whose
determinant has at most two zeros. For a remaining line, the first
generalized Bockstein map controls mod-five co-descent. Integral co-descent,
localization, and the motivic genus formula determine the five-primary
tame kernel and make the integral restriction cokernel vanish. Finally, an
explicit integral dihedral-lattice calculation converts that vanishing
into the Bloch lattice index five.

The argument does not require a distribution theorem for symbols attached
to generators of varying prime ideals. Determining how often all four
mixed cup products are surjective is a separate problem and is not used
in the theorem or in Corollary~\ref{cor:count}.

\section{The fixed extensions and their Chebotarev set}
We use the integral weight-two comparison for odd-primary tame kernels,
via the comparison in \cite[Section 1.3]{AM19} and localization as in
\cite[Section 1.1]{AM12}. Explicitly, if
\(S\) contains the five-primes and the infinite places, then
\[
 K_2(O_J[S^{-1}])(5)\simeq H^2(G_{J,S},\Z_5(2)).
\]
For \(S\) consisting only of these places, inverting the five-primes does
not change the five-primary tame kernel: the removed residue groups have
orders \(5^f-1\), prime to five. We also use Levine's natural comparison
\cite[Theorem 4.5]{Levine} between completed indecomposable \(K_3\) and
weight-two integral \(H^1\).

\begin{lemma}\label{lem:base}
\(K_2(\Z[i])(5)=0\), and \(H^2(G_{k,S_0},\F_5(2))=0\).
\end{lemma}
\begin{proof}
The field \(E=\Q(\zeta_{20})\) has class number one; the accompanying
fixed-field certificate verifies this fact independently. Kummer theory gives
\(H^2(O_E[1/5],\mu_5)=\operatorname{Br}(O_E[1/5])[5]\).
The Brauer invariant sequence identifies this group with the kernel of
\(\F_5^2\xrightarrow{+}\F_5\), one coordinate for each five-prime.
Both primes are fixed by \(\Delta=\Gal(E/k)\), and local invariants show
that \(\Delta\) acts trivially on this kernel. Twisting by
\(\F_5(1)\) therefore gives a nontrivial cyclotomic character on
\(H^2(G_{E,S_0},\F_5(2))\). Its \(\Delta\)-invariants vanish.
Since \(|\Delta|=4\), Hochschild--Serre identifies those invariants with
\(H^2(G_{k,S_0},\F_5(2))\). The integral weight-two comparison with the
tame kernel, reduction modulo five, and finiteness then give the asserted
five-primary vanishing.
\end{proof}

The five-primary character space of the ray class group of \(k\) modulo
\((25)\) has dimension two. Complex conjugation exchanges the two primes above
five. Its negative eigenspace has dimension one. Fix a nonzero character
\(\eta_0\) in this eigenspace, with fixed field \(M_0\). Thus \(M_0/k\) is
cyclic of degree five and conductor \((25)\), and \(M_0/\Q\) is dihedral.
This uses \(h(k)=1\) and \(|\mu(k)|=4\).

Let \(T=\F_5(2)\). The integral group
\(H^1(G_{k,S_0},\Z_5(2))\) is free of rank one over \(\Z_5\): its rank is
the Borel rank \(r_2(k)=1\), and its torsion is
\(H^0(k,\Q_5/\Z_5(2))=0\), since \(\sqrt5\notin k\).
The vanishing of the integral \(H^2\) and Bockstein give
\[
 H^1(G_{k,S_0},\Z_5(2))/5\simeq H^1(G_{k,S_0},T).
\]
Restriction to \(E\) is injective, because \([E:k]=4\) is prime to five.
The Kummer comparison
\(H^1(G_{E,S_0},T)\simeq H^1(G_{E,S_0},\mu_5)\otimes\F_5(1)\)
turns its image into a
one-dimensional nonzero subgroup \(D\subset E^\times/E^{\times5}\).
This is the weight-two Tate line of \cite{AM19}. It is unramified outside five
in the sense that \(E_T=E(\sqrt[5]{D})/E\) is unramified outside five.

Let \(\Delta=\Gal(E/k)\simeq C_4\), and let \(\omega\) be its cyclotomic
character modulo five. The line \(D\) has character \(\omega^{-1}\).
Consequently, the Kummer pairing with target \(\mu_5\) shows that
\(\Delta\) acts on \(\Gal(E_T/E)\) with character \(\omega^2\).
On \(\Gal(M_0E/E)\), the action is trivial, since \(M_0E/k\) is abelian.
Thus the two cyclic quintic extensions of \(E\) are different and linearly
disjoint. The Tate line is also stable under conjugation over \(\Q\);
both factors, and hence
\[
 L=E_T M_0,
\]
are normal over \(\Q\). We have \([L:\Q]=8\cdot25=200\).

Let \(\mathcal P\) consist of the unramified rational primes whose Frobenius
conjugacy class in \(L/\Q\) lies in
\(\{0\}\times(C_5\setminus\{0\})\subset\Gal(L/E)\), with coordinates
corresponding to \(E_T/E\) and \(M_0E/E\), respectively.
In \(\Gal(L/E)=C_5\times C_5\), this condition selects exactly four elements.
It is invariant under conjugation, because both factors are normal over
\(\Q\). Chebotarev gives density \(4/200=1/50\).
Such primes split completely in \(E\), so \(\ell\equiv1\pmod{20}\).
For either prime of \(k\) above \(\ell\), the unramified character
\(\eta_0\) has nonzero Frobenius value, and the Tate Frobenius is zero.
Different choices of a prime of \(E\) change Frobenius by a nonzero scalar
conjugation; its vanishing and the dimension of its span are unchanged.

\section{A cup-product pencil on the fixed base field}
Fix \(\ell\in\mathcal P\), and let \(v_1,v_2\) be the primes of \(k\) over
\(\ell\). The anti-invariant order-five character space modulo \((\ell)\)
is one-dimensional. Choose a nonzero character \(\eta_\ell\) there.
For \(\lambda\in\F_5^\times\), put
\[
 \eta_\lambda=\eta_0+\lambda\eta_\ell,
 \qquad M_\lambda=\overline{k}^{\ker\eta_\lambda}.
\]
The two characters have different ramification sets and are independent.
These are four different cyclic quintic layers, each fully ramified at both
five-primes and both \(\ell\)-primes and with conductor \((25\ell)\).
Conjugation negates each character, so each field is dihedral over \(\Q\).
The conductor--discriminant formula gives
\[
 \operatorname{disc}(M_\lambda)
 =(-4)^5\N_{k/\Q}((25\ell)^4)
 =-2^{10}5^{16}\ell^8.
\]

Set \(\mathcal G=G_{k,S}\), with \(S=S_0\cup\{v_1,v_2\}\).
The base group \(H^2(G_{k,S_0},T)\) is zero. Localization and purity give
\begin{align}
 H^1(\mathcal G,T)&\longrightarrow
 \bigoplus_{j=1}^2 H^0(\kappa(v_j),\F_5(1))\longrightarrow0,\label{eq:inertia}\\
 H^2(\mathcal G,T)&\xrightarrow{\sim}
 \bigoplus_{j=1}^2 H^2(k_{v_j},T)\simeq\F_5^2.\label{eq:h2loc}
\end{align}
In \eqref{eq:inertia}, each summand is \(\F_5\), since
\(\ell\equiv1\pmod5\), and the map records tame inertia residues.
Surjectivity follows from the preceding zero \(H^2(G_{k,S_0},T)\).
For \eqref{eq:h2loc}, localization identifies the new finite-field
\(H^1(\kappa(v_j),\F_5(1))\) with the local \(H^2\); the next global
\(H^3\) is zero. The five-adic local \(H^2\) groups vanish because those
completions do not contain \(\mu_5\).

Choose \(y_1,y_2\in H^1(\mathcal G,T)\) mapping to the standard basis in
\eqref{eq:inertia}. The character \(\eta_0\) is unramified at these primes,
with Frobenius values \(u,-u\), where \(u\ne0\). Locally, cup product pairs
an unramified \(\F_5\)-character perfectly with the one-dimensional tame
inertia residue in \(H^1(T)\); the cup of two unramified classes is zero.
Consequently \(\eta_0\cup-\), restricted to the span of \(y_1,y_2\), has
an invertible diagonal matrix \(C_0\) in \eqref{eq:h2loc}.

Let \(C_\ell\) be the matrix for \(\eta_\ell\cup-\) on the same two
vectors. For \(\eta_\lambda\), the matrix is
\[
 C_\lambda=C_0+\lambda C_\ell.
\]
Its determinant is a polynomial of degree at most two over \(\F_5\), with
nonzero constant coefficient. It has at most two roots. Therefore at least
two of the four nonzero \(\lambda\) satisfy
\begin{equation}\label{eq:cup}
 \eta_\lambda\cup-:H^1(\mathcal G,T)\twoheadrightarrow H^2(\mathcal G,T).
\end{equation}
No assertion about the distribution of residue symbols of varying prime
generators is required: \(C_\ell\) may depend arbitrarily on \(\ell\).

\section{From cup product to integral co-descent}
Fix a line satisfying \eqref{eq:cup}. Write \(M=M_\lambda\),
\(\mathcal N=\ker\eta_\lambda=G_{M,S}\), and
\(G=\mathcal G/\mathcal N=\langle\rho\rangle=C_5\).
Put \(\Omega=\F_5[G]\), \(I=(\rho-1)\), and \(A=H^2(\mathcal N,T)\).
Theorem 2.2.4 of Lam--Liu--Sharifi--Wake--Wang \cite{LLSWW}, with
\(R=\F_5\), \(n=1\), and finite \(H=G\), gives
\[
 IA/I^2A\simeq\coker\bigl(H^1(\mathcal G,T)
 \xrightarrow{\,\eta_\lambda\cup-\,}H^2(\mathcal G,T)\bigr).
\]
The group \(\mathcal G\) has five-cohomological dimension two and finite
cohomology with finite five-primary coefficients. The coefficient ring
\(\F_5\) is complete local Noetherian, and the finite group \(C_5\) is a
compact five-adic Lie group. Thus the hypotheses of \cite[Section 2.2]{LLSWW}
hold; the required flatness is automatic. We identify \(I/I^2\) with
\(\F_5\) using the class of \(\rho-1\).
The identification of \(\mathcal N\) follows because \(M/k\) is unramified
outside \(S\), and taking the \(k\)-normal closure of an extension of
\(M\) unramified outside \(S\) introduces no ramification outside \(S\).
For finite \(H\), its Iwasawa cohomology is ordinary subgroup cohomology
\cite[Remark 2.2.1]{LLSWW}; the first Bockstein is the cup product
\cite[Proposition 2.3.3]{LLSWW}.
Equation \eqref{eq:cup} gives \(IA=I^2A\). Since \(I^5=0\), it gives
\(IA=0\). Top-degree corestriction gives
\(A_G\simeq H^2(\mathcal G,T)\), by \cite[Equation (2.4)]{LLSWW}.
Consequently \(A\simeq\F_5^2\).

Put
\[
 B=H^2(G_{M,S},\Z_5(2)),\quad W=K_2(O_M)(5),
 \quad a=v_5(\ell-1),\quad R=(\Z/5^a)^2.
\]
These integral cohomology groups are finite. Cohomological dimension at most
two and the Mittag--Leffler property for finite coefficient groups give
continuous \(H^3(G_{M,S},\Z_5(2))=0\). Thus Bockstein gives
\(B/5B\simeq A\), so the minimal number of \(\Z_5\)-generators of
\(B\) is two.

The integral localization sequence is
\begin{equation}\label{eq:integralloc}
 0\longrightarrow W\longrightarrow B
 \xrightarrow{\lambda_{\mathrm{loc}}}R\longrightarrow0.
\end{equation}
The base group \(H^2(G_{k,S},\Z_5(2))\) is also \(R\), because its tame
kernel has zero five-primary part. Integral co-descent
\cite[Theorem 1.5 and the following paragraph]{AM12} identifies
\(\cor:B_G\xrightarrow{\sim}R\).
The localization diagram is natural: on residue multiplicative groups the
corestriction is the residue-field norm, while restriction is the
ramification-index power. Here \(f=1\) and \(e=5\), so, in matched residue
coordinates,
\begin{equation}\label{eq:transfer}
 q:=\cor=\lambda_{\mathrm{loc}},\qquad
 (\rho-1)B=W,\qquad \res\,q=N_G.
\end{equation}
More explicitly, these maps occur in the commutative diagram
\[
\begin{array}{ccccccccc}
0&\longrightarrow&W&\longrightarrow&B&\longrightarrow&R_M&\longrightarrow&0\\
 &&\downarrow&&\downarrow\cor&&\downarrow\N_{\rm res}\\
0&\longrightarrow&0&\longrightarrow&H^2(G_{k,S},\Z_5(2))&\xrightarrow{\sim}&R_k&\longrightarrow&0.
\end{array}
\]
Here \(R_M\) and \(R_k\) are the two residue terms, both identified with
\(R\). Their residue fields agree, so \(\N_{\rm res}\) is the identity.
Restriction on these terms is multiplication by five, since the
ramification index is five.

\begin{lemma}\label{lem:integral}
With the preceding notation, \(W=5^aB\), \(5^{a+1}B=0\),
\(N_G=5\operatorname{id}_B\), and \(G\) acts trivially on \(W\).
\end{lemma}
\begin{proof}
A surjection from a finite abelian five-group with two generators onto
\((\Z/5^a)^2\) forces both cyclic factors to have exponent at least \(a\).
Hence \(|B/5^aB|=|R|\). Since \(5^aB\subset\ker q\), comparison of orders
in \eqref{eq:integralloc} gives \(W=5^aB\).
Let \(t=\rho-1\). Then \(tB=5^aB\) and \(t^jB=5^{aj}B\).
The identity
\[
 0=5t+10t^2+10t^3+5t^4+t^5
\]
implies \(5^{a+1}B\subset5^{a+2}B\), as every other term has image contained
in the latter group when \(a\ge1\). Finiteness forces \(5^{a+1}B=0\).
Thus \(t^2B=0\); in
\(N_G=5+10t+10t^2+5t^3+t^4\), all nonconstant terms vanish, including
\(10tB\subset5^{a+1}B\). Finally
\(tW=5^atB=5^{2a}B=0\).
\end{proof}

\section{The genus formula and the actual Bloch lattice}
The motivic tame-kernel genus formula
\cite[Proposition 4.11 and Corollary 4.12]{AM19} gives
\[
 \frac{|W_G|}{|K_2(O_k)(5)|}=5^{2-t_2}.
\]
Here \(t_2\) is the dimension of the span of the two tame Frobenius elements
in the fixed extension \(E_T/E\). It depends only on the weight-two Tate line
of \(k\) and the two tame ramified primes. All four mixed lines have the same
tame set and nonzero local characters; changing \(\lambda\) multiplies a tame
character by a nonzero scalar and does not change its norm kernel.
The extension is cyclic of prime degree five, and \(i=2\) in the cited
formula. Its two tame ramified primes are \(v_1,v_2\); the wild five-primes
belong to \(S_5\) and are excluded from \(S\setminus S_5\). Proposition
4.11 identifies the norm index with the Frobenius-span dimension \(t_2\).
Both Frobenius elements vanish by the definition of \(\mathcal P\), so
\(t_2=0\) for all four mixed extensions.

Lemma \ref{lem:integral} makes \(W_G=W\). The genus formula gives \(|W|=25\),
and \(5W=0\). Both cyclic factors of \(B\) must consequently have exponent
\(a+1\). The equality \(\res q=5\operatorname{id}_B\), with \(q\) surjective,
gives \(\res(R)=5B\). These groups have equal order, so restriction is
injective. For a finite cyclic action \(|B^G|=|B_G|=|R|\), so restriction is
also surjective onto \(B^G\). In particular the integral restriction
\begin{equation}\label{eq:f2}
 f_2:H^2(G_{k,S},\Z_5(2))\longrightarrow B^G
\end{equation}
has zero cokernel.

We include the lattice consequence to make the link to \(Q_2\) explicit.
It uses no special-value conjecture.
\begin{proposition}\label{prop:lattice}
Let \(M/\Q\) be a non-real \(D_{10}\)-extension and, within this
proposition, put \(k=M^{C_5}\). Let \(S\) contain the infinite places, the
five-primes, and all primes ramified in \(M/k\). If the restriction map
\[
 f_2:H^2(G_{k,S},\Z_5(2))\longrightarrow H^2(G_{M,S},\Z_5(2))^{C_5}
\]
has zero cokernel, then \(Q_2(M)=5\).
\end{proposition}
\begin{proof}
Let \(\Lambda=K_3^{\mathrm{ind}}(M)\otimes\Z_5\).
It is a free lattice of rank five. The torsion formula for indecomposable
\(K_3\) and the Chern comparison \cite[Theorem 4.5 and Corollary 4.6]{Levine}
show that five-torsion would require \(\Q(\sqrt5)\subset M\).
But the unique quadratic subfield of a non-real dihedral field is imaginary.
Thus there is no five-torsion.
Levine's integral Galois descent \cite[Theorem 4.13]{Levine} identifies the
actual subfield images with \(\Lambda^{C_5}\) and the two reflection invariant
lattices. Its finite-quotient order condition is prime to the field
characteristic, which is zero; it does not exclude degree five.
The natural Suslin exact sequence identifies the free \(K_3\) quotient with
the free Bloch quotient, compatibly with these images.

Weight-two integral \(H^1\) does not increase on enlarging \(S\), because
the possible new finite-field \(H^0(\kappa(v),\Z_5(1))\) is zero.
Thus Levine's comparison identifies
\(H^1(G_{M,S},\Z_5(2))\) naturally with \(\Lambda\), compatibly with
the Galois action. Theorem 1.5(ii) of \cite{AM12} identifies the restriction
cokernel with the second group cohomology of this module. Cyclic
periodicity then gives
\[
 \coker f_2\simeq H^2(C_5,\Lambda)
 \simeq\Lambda^{C_5}/N_G\Lambda=0.
\]
Write \(\Lambda^{C_5}=\Z_5y\). This invariant sublattice is saturated,
so \(P=\Lambda/\Z_5y\) is free of rank four over \(\Z_5\).
The norm \(N_G=\Phi_5(\rho)\) annihilates \(P\); hence \(P\) is a
torsion-free rank-one module over the discrete valuation ring
\(\Z_5[\zeta_5]\), and is free. Lift a generator to \(x\), with
\(N_Gx=by\). Since \(y,x,\rho x,\rho^2x,\rho^3x\) generate \(\Lambda\),
we have \(N_G\Lambda=(5,b)\Z_5y\). The zero norm quotient forces
\(b\in\Z_5^\times\), whence \(y=b^{-1}N_Gx\). Thus \(x\) generates
\(\Lambda\) over \(\Z_5[C_5]\). The resulting surjection is an
isomorphism because both sides have \(\Z_5\)-rank five.

Let \(c\) be a reflection. On the invariant line it acts as \(-1\), since
that line descends to the imaginary quadratic field, whose \(K_3\) free
quotient has no rational invariant part. In the regular group ring
\(\mathcal R=\Z_5[C_5]\), write \(cx=wx\), so \(w\bar w=1\) and
\(\operatorname{aug}(w)=-1\), with \(\bar\rho=\rho^{-1}\).
The element \(u=(1-w)/2\) is a unit: it has augmentation one in the
local ring \(\Z_5[C_5]\). It satisfies \(w\bar u=-u\). Replacing \(x\) by \(ux\)
normalizes the action to \(c(rx)=-\bar r x\).

For a basis \(e_j=\rho^jx\), a reflection \(\rho^bc\) sends
\(e_j\) to \(-e_{b-j}\). Its invariant lattice is generated by the
differences \(e_j-e_{b-j}\). Two different reflections give a connected
graph on the five indices: their composition is a nonzero translation.
Their differences therefore generate the augmentation ideal \(I_{\mathcal R}\).
Adding \(\Lambda^{C_5}=\Z_5N_Gx\), the quotient is
\(\mathcal R/(I_{\mathcal R}+\Z_5N_G)\simeq\Z/5\).
This computes the five-adic index of the actual subfield Bloch lattices.
Guo--Qin's bound \(Q_2\mid5\) makes it the full index.
\end{proof}

Apply Proposition \ref{prop:lattice} to the good lines to finish
Theorem \ref{thm:main}. Varying \(\ell\) changes the discriminant, so the
fields obtained from distinct parameter primes are non-isomorphic.

\begin{corollary}\label{cor:seed}
The conductor-25 seed \(M_0\) satisfies
\(K_2(O_{M_0})(5)=0\) and \(Q_2(M_0)=5\).
\end{corollary}
\begin{proof}
Put \(B_0=H^2(G_{M_0,S_0},\Z_5(2))\). Top-degree integral co-descent
gives \((B_0)_{C_5}=0\), by Lemma~\ref{lem:base}. The group \(B_0\)
is a finite five-group; the co-descent isomorphism is the one following
\cite[Theorem 1.5]{AM12}. Reducing modulo five, the augmentation ideal of
\(\F_5[C_5]\) is nilpotent, so vanishing of coinvariants forces
\(B_0/5B_0=0\), and hence \(B_0=0\). Localization gives the tame-kernel
assertion. Restriction is consequently \(0\to0\), and
Proposition~\ref{prop:lattice} gives the index assertion.
\end{proof}
The seed already realizes the nontrivial index. The main theorem adds
infinitely many extensions over the same quadratic field, mixed tame and
wild ramification, and the specified nonzero tame kernel.

\section{Fixed data and reproducibility}
The theorem uses a fixed finite extension to define its prime set.
It also has a small-field exact description. In \(E=\Q(y)\), \(y=\zeta_{20}\),
one may take the Tate-line generator
\[
 g_2=-5y^7+7y^6-3y^5+y^4-5y^3+6y^2-3.
\]
The supplied PARI/GP 2.15.5 certificate proves \(h(E)=1\), certifies the
fundamental unit group, and records \(g_2\)'s coordinates \((1,3,1,0)\)
in the basis
\[
 (u_1,u_2,u_3)=(y^4+1,y-1,y^3-1),\qquad \mu(E)=\langle y\rangle.
\]
Thus it is not a fifth power. If \(\sigma_{13}(y)=y^{13}\) and
\(c(y)=y^{-1}\), the unit coordinates of
\(\sigma_{13}(g_2)/g_2^2\) and \(c(g_2)/g_2\) are respectively
\((0,-5,-5,15)\) and \((0,0,0,10)\), where the final coordinate refers to
the order-twenty roots of unity. More explicitly, for
\(h=y^7-y^5-y^2-y-1\), one has the polynomial identities
\[
 \sigma_{13}(g_2)=g_2^2h^5,\qquad c(g_2)=g_2(y^2)^5
 \quad\text{in }\Q[y]/(\Phi_{20}(y)).
\]
The ancillary file \texttt{anc/check\_tate\_identities.gp} verifies these
identities and prints the full unit basis. This gives the
required \(\Delta\)-character \(\omega^{-1}\). Since it is an \(S_0\) Kummer
class and the base cohomology is one-dimensional, it generates the canonical
line \(D\).

The fixed pure-wild field \(M_0\) is the splitting field of
\[
 f_0(X)=X^5+10X^3-10X^2-15X-18.
\]
A separate ray-class certificate over \(\Q(i)\) proves its group,
conductor, and discriminant. For \(\ell\equiv1\pmod{20}\), membership in
\(\mathcal P\) can be checked by
\[
 g_2^{(\ell-1)/5}\equiv1\pmod Q\quad(Q\mid\ell\text{ in }E),
 \qquad f_0\bmod\ell\text{ irreducible}.
\]
The second test detects the nontrivial cyclic Frobenius, since \(\ell\) splits
in the quadratic subfield. The ancillary file \texttt{anc/check\_tate\_data.gp} checks all eight residue
symbols in \(E\), using PARI/GP 2.15.5. It finds that \(701,941,1321,1381,1801,2081\) meet both
conditions, whereas \(41\) fails the Tate condition. These calculations check
fixed data and examples; the infinite-family proof uses Chebotarev and the
cup-product argument. The file \texttt{anc/check\_wild\_seed.gp}
constructs the seed and checks its group, conductor, and discriminant.
Both files use \texttt{bnfcertify}; their complete outputs are supplied.
The success markers are \texttt{TATE\_DATA\_DONE} and
\texttt{WILD\_SEED\_DONE}. These checks certify the fixed arithmetic
inputs, rather than the index of a selected mixed extension.

\section{Relation to earlier work and scope}
Browkin--Gangl \cite[Section 12.4]{BG} already gave numerical evidence for
\(Q_2=5\), including a dihedral Hilbert-class-field example with quadratic
subfield \(\Q(\sqrt{-47})\). Their numerical lattice-generation qualifications
must be kept distinct from an integral proof. Guo--Qin \cite{GQ} established
the dihedral divisibility theorem, which we use, and an infinite family with
\(Q_2=1\) and quadratic subfield \(\Q(\sqrt{-3})\).
Caputo \cite{Caputo12} gives algebraic Brauer--Kuroda formulae for
dihedral extensions in terms of integral higher-unit indices, without
assuming a Lichtenbaum special-value formula. Regulator constants and
cohomological index formulae are further developed in
\cite{Caputo26,BS}. These provide the general integral framework; the
arithmetic input here is the fixed-base mixed ray construction and its
uniform cup-product selection. The general integral
co-descent and genus statements used here belong to Assim--Movahhedi
\cite{AM12,AM19}; the generalized Bockstein theorem belongs to
Lam--Liu--Sharifi--Wake--Wang \cite{LLSWW}.

The arithmetic construction here keeps the quadratic subfield fixed,
introduces mixed tame and wild ramification, and selects among four actual
extensions with a uniform two-root bound. Corollary~\ref{cor:seed} separates
the simpler regularity-based seed example from this mixed infinite family.

The construction leaves open the distribution of the four mixed cup
products. The determinant argument guarantees at least two surjective
maps for every parameter in \(\mathcal P\), but does not determine how
often three or all four are surjective.

\paragraph{Use of generative AI.}
Generative AI tools, including OpenAI Codex and Kimi, were used extensively
in the research and preparation of this manuscript. Their uses included
exploring proof strategies, developing and revising arguments, assisting
with literature searches and checks of cited results, producing and
checking PARI/GP scripts, and drafting and revising the text. AI-assisted
critiques were also used during revision; these do not constitute external
peer review. The author remains responsible for all mathematical claims,
computations, references, and the final manuscript.

\end{document}